\documentclass[12pt]{amsart}
\usepackage{amssymb}
\usepackage{mathptmx}

\let\oldlabel=\label
\def\prellabel{\marginparsep=1em\marginparwidth=44pt
\def\label##1{\oldlabel{##1}\ifmmode\else\ifinner\else
\marginpar{{\footnotesize\ \\ \tt ##1}}\fi\fi}}

\usepackage[latin1]{inputenc}
\usepackage[all]{xy}

\newcommand{\mm}{\mathfrak m}
\newcommand{\N}{\mathbb{N}}

\newcommand{\Z}{\mathbb{Z}}
\DeclareMathOperator{\Ann}{Ann} 
 
\DeclareMathOperator{\chara}{char}

 \DeclareMathOperator{\GL}{GL}

\DeclareMathOperator{\Image}{Im}

\DeclareMathOperator{\reg}{reg}

\DeclareMathOperator{\signum}{sign}

\DeclareMathOperator{\Tor}{Tor} 
 \DeclareMathOperator{\ind}{index}

\DeclareMathOperator{\fc}{c} \DeclareMathOperator{\ffb}{b}

\DeclareMathOperator{\dirsum}{\oplus}

\DeclareMathOperator{\pnt}{\raise 0.5mm \hbox{\large\bf.}}

\newtheorem{thm}{\bf Theorem}[section]
\newtheorem{lem}[thm]{\bf Lemma}
\newtheorem{cor}[thm]{\bf Corollary}
\newtheorem{prop}[thm]{\bf Proposition}

\newtheorem{quest}[thm]{\bf Question}

\theoremstyle{definition}

\newtheorem{rem}[thm]{\bf Remark}

\let\signum=\relax

\title{Koszul cycles}

\author{Winfried Bruns}
\address{Universit\"at Osnabr\"uck, Institut f\"ur Mathematik, 49069 Osnabr\"uck, Germany}
\email{wbruns@uos.de}

\author{Aldo Conca}
\address{Dipartimento di Matematica, Universit\'{a} di Genova, Via Dodecaneso 35, 16146 Genova,
Italy} \email{conca@dima.unige.it}

\author{Tim R\"omer}
\address{Universit\"at Osnabr\"uck, Institut f\"ur Mathematik, 49069 Osnabr\"uck, Germany}
\email{troemer@uos.de}

\dedicatory{ }

\begin{document}

\begin{abstract}
We prove regularity bounds for Koszul cycles holding for every ideal
of dimension $\leq 1$ in a polynomial ring; see Theorem \ref{thm1}.
In Theorem \ref{check} we generalize the ``$c+1$" lower bound for
the Green-Lazarsfeld index of Veronese rings proved in \cite{BCR} to
the multihomogeneous setting. For the Koszul complex of the $c$-th
power of the maximal ideal in a Koszul ring we prove that the cycles
of homological degree $t$ and internal degree $\geq t(c+1)$ belong
to the $t$-th power of the module of $1$-cycles; see Theorem
\ref{maincyc}.

\end{abstract}


\maketitle

%
%
%
\section{Introduction}
\label{sec:intro} The Koszul complex and its homology are central
objects in commutative algebra. Vanishing theorems for Koszul
homology are the key to many open questions. The goal of the paper
is the study of regularity bounds  for Koszul cycles and Koszul
homology of ideals in standard graded rings. Our  original
motivation comes from the study of the syzygies of Veronese
varieties and, in particular, the conjecture of Ottaviani and
Paoletti \cite{OP} on their Green-Lazarsfeld index, see \cite{BCR}.

In Section \ref{Canonical maps} we fix the  notation and describe
some canonical maps between modules of Koszul cycles. Given  a
standard graded ring $R$ with maximal homogeneous ideal $\mm$, a
homogeneous ideal $I$ and a finitely generated graded module $M$, we
let $Z_t(I,M)$ denote the module of Koszul cycles of homological
degree $t$. Under a mild assumption, we show  in \ref{dirsum}  that
$Z_{s+t}(I,M)$ is a direct summand of  $Z_{s}(I,N)$  where
$N=Z_t(I,M)$.

Section \ref{Bang}  is devoted to the description of
(Castelnuovo-Mumford) regularity bounds for Koszul cycles and
homology.  We prove bounds of the following type:
\begin{equation}
\label{mamain} \reg_R(Z_t(I,M))\leq t(c+1)+\reg_R(M)+v
\end{equation}
under assumptions on $\dim M/IM$. Here $\reg_R(N)$ denotes the
(relative) Castelnuovo-Mumford regularity of a finitely generated
$R$-module $N$.   Note that $\reg_R(N)$ is the ordinary
Castelnuovo-Mumford regularity if $R$ is the polynomial ring.
Furthermore it is known that  $\reg_R(N)$ is finite if $R$ is a
Koszul algebra. If $R$ is Koszul and $\dim M/IM=0$, then we prove
that (\ref{mamain}) holds with $v=0$ and where $c$ is such that
$\mm^c\subset I+\Ann(M)$ and $I$ is generated in degrees $\leq c$, see \ref{regb1}.  
In \ref{thm1} we prove that if $R$ is a polynomial ring of characteristic $0$ or big enough
and $\dim M/IM\leq 1$ then   (\ref{mamain}) holds with $v=0$ and
$c=\reg_R(I)$.  Furthermore, if $R$ is a polynomial ring and $\dim
M/IM=0$, then we show in \ref{greeny} that (\ref{mamain}) holds with
$c\geq $ the largest degree of a generator of $I$ and $v=\dim
[R/I]_c$.

 We also give examples showing that  the inequality
\begin{equation}
\label{mamain2} \reg_R Z_t(I,M)\leq t(\reg_R(I)+1)+\reg_R(M)
\end{equation}
cannot hold in general (i.e. without restriction on the dimension of
$M/IM$). However (\ref{mamain2})  holds if  $R$ is a polynomial
ring, $M=R/J$ and both $I$ and $J$ are strongly stable monomial
ideals, see \ref{piper} and \ref{sato}. We leave it as an open
question  whether (\ref{mamain2}) holds when $M=R$ and $R$ is a
polynomial ring.

In Section  \ref{SeVe} we prove that, given a vector
$c=(c_1,\dots,c_d)\in \N_+^d$, the Segre-Veronese ring associated to
$c$ over a field of characteristic $0$ or big enough, has a
Green-Lazarsfeld index larger than or equal to $\min(c)+1$,  see
\ref{check}. This result was announced in \cite{BCR} and improves
the bound of Hering, Schenck and Smith \cite{HSS} by $1$.

In Section \ref{Gency} we analyze  the generators  of the module
$Z_t=Z_t(\mm^c,R)$ under the assumption that $R$ has characteristic
$0$ or big enough.   If $R$ is Koszul  we prove that $Z_t/ Z_1^t$
vanishes  in degrees $\geq t(c+1)$, \ref{maincyc}.  Here $Z_1^t$
denotes the image of the canonical map $\wedge^t Z_1\to Z_t$. This
allows us to deduce that the $c$-th Veronese subring of a polynomial
ring $S$ satisfies the property $N_{2c}$ if and only if
$H_1(\mm^c,S)^{2c}=0$, see \ref{prcpt}.  Finally, we prove that  the
cycles given in \cite{BCR}  generate   $Z_2$; see \ref{gen2}.

\section{Notation and generalities}
\label{Canonical maps} In this section we collect notation and
general facts about maps between modules of Koszul cycles. Let $R$
be a ring, $F$ be a free $R$-module of rank $n$, $\phi:F\to R$ be an
$R$-linear map and $M$ be an $R$-module. All tensor products are
over $R$. We consider the Koszul complexes
$K(\phi,R)=\bigoplus_{t=0}^n K_t(\phi,R)=\bigwedge^\bullet F$ and
$K(\phi,M)=\bigoplus_{t=0}^nK_t(\phi,M)=\bigwedge^\bullet F\otimes
M$. The complex  $K(\phi,M)$ can be seen as a module over the
exterior algebra $K(\phi,R)$. For $a\in K(\phi,R)$ and $f\in
K(\phi,M)$ the multiplication will be denoted by $a.f$. The
differential of $K(\phi,R)$ and $K(\phi,M)$ will be denoted simply
by $\phi$ and it satisfies
$$\phi(a.f)= \phi(a).f + (-1)^s a.\phi(f)$$
for all $a\in K_s(\phi,R)$ and $f\in K(\phi,M)$. We let
$Z_t(\phi,M)$, $B_t(\phi,M)$, $H_t(\phi,M)$  denote the cycles, the
boundaries and the homology in homological degree $t$ and set
$Z(\phi,M)=\dirsum Z_t(\phi,M)$, and so on for cycles, boundaries
and homology. One knows that $Z(\phi,R)$ is a subalgebra of
$K(\phi,R)$ and that $B(\phi, R)$ is an ideal of $Z(\phi,R)$ so that
the homology $H(\phi,R)$ is itself an algebra. More generally,
$Z(\phi,M)$ is a $Z(\phi,R)$-module. We let $Z_s(\phi,R)Z_t(\phi,M)$
denote the image of the multiplication map $Z_s(\phi,R)\otimes
Z_t(\phi,M)\to Z_{s+t}(\phi,M)$. Similarly, $Z_1(\phi,R)^t$ will
denote the image of the map $\bigwedge^t Z_1(\phi,R)\to
Z_{t}(\phi,R)$.

In the graded setting the map $\phi$ will be assumed to be of degree
$0$ and $F$ will be a direct sum of shifted copies of $R$. In this
way the Koszul complex $K(\phi,M)$ inherits a graded structure for
the map $\phi$ and the module $M$. So cycles, boundaries and
homology have an induced graded structure. An index on the left of a
graded module always denotes the selection of the homogeneous
component of that degree.  If $R$ is standard graded over a field
$K$ with maximal homogeneous ideal $\mm$  all the invariants we are
going to study depend actually only on the image of $\phi$ and not
on the map itself as long as $\ker \phi\subseteq \mm F$. So, if
$J=\Image \phi$, we will sometimes denote $K(\phi,R)$ simply by
$K(J,R)$ and so on.

Fix a basis of the free module $F$, say $\{e_1,\dots,e_n\}$. Given
$I=\{i_1,\dots,i_s\}\subset [n]$ with $i_1<i_2<\dots<i_s$ we write
$e_I$ for the corresponding basis element $e_{i_1}\wedge \cdots
\wedge e_{i_s}$ of $\bigwedge^s F$. If $\phi(e_i)=u_i\in J$ we will
also use the symbol $[u_{i_1},\dots,u_{i_s}]$ to denote $e_I$.

For disjoint subsets $A,B\subset [n]$ we set $\epsilon(A,B)=\#\{
(a,b)\in A\times B : a>b\}$ and
$$\sigma(A,B)=(-1)^{\epsilon(A,B)}.$$
One has
$$e_Ae_B=\sigma(A,B)e_{A\cup B}.$$
For further application we record the following:

\begin{lem}
\label{checksign} For disjoint subsets $A,B,C$ of $[n]$ one has
$$\sigma(A\cup B,C)\sigma(B,A)=\sigma(B,A\cup C)\sigma(A,C).$$
\end{lem}

\begin{proof} Just use the fact that
$\epsilon(A\cup B,C)=\epsilon(A,C)+\epsilon(B,C)$ and
$\epsilon(B,A\cup C)=\epsilon(B,A)+\epsilon(B,C)$.
\end{proof}

Any element $f\in \bigwedge^s F\otimes M$ can be written uniquely as
$f=\sum e_I\otimes m_I$ with $m_I\in M$ where the sum is over the
subsets of cardinality $s$ of $[n]$. If $m_I= 0$ then we will say
that $e_I$ does not appear in $f$. For every $f\in K_{s+t}(\phi, M)$
and for every $I\subset [n]$ with $s=\# I$ we have a unique
decomposition
\begin{equation}
\label{deco} f=a_I+ e_I.b_I
\end{equation}
with $a_I\in K_{s+t}(\phi, M)$ and $b_I\in K_{t}(\phi, M)$, and,
furthermore, $e_J$ does not appear in $a_I$ whenever $J\supset I$
and $e_S$ does not appear in $b_I$ whenever $S\cap I\neq \emptyset$.
With the notation above we have:

\begin{lem}\label{le1}
For every $f\in K_{s+t}(\phi, M)$ we have:
\begin{itemize}
\item[(a)] $\sum_I e_I.b_I =\binom{t+s}{s} f$ where $\sum_I$ stands
for the sum extended to all the subsets $I\subset [n]$ with $s=\#
I$.
\item[(b)] if $f\in Z_{s+t}(\phi, M)$, then $b_I\in Z_{t}(\phi, M)$
for every $I$ with $s=\# I$.
\end{itemize}
\end{lem}

\begin{proof} For (a) one writes $f=\sum e_J \otimes m_J$ with $J\subset [n]$
with $\#J=s+t$ and $m_J\in M$. Then one observes that $e_J.m_J$
appears in $e_I.b_I$ iff $I\subset J$. Hence $e_J.m_J$ appears in
$\sum_I e_I.b_I$ exactly $\binom{t+s}{s}$ times. For (b) one applies
the differential $0=\phi(f)=\phi(a_I)+ \phi(e_I).b_I +(-1)^s
e_I.\phi(b_I)$ and since $e_J$ does not appear in $\phi(a_I)+
\phi(e_I).b_I$ whenever $J\supseteq I$ then $\phi(b_I)$ must be $0$.
\end{proof}

The multiplication $K_s(\phi,R)\otimes K_t(\phi,M)\to K_{s+t}(\phi,
M)$ can be interpreted as a map
$$K_s (\phi, K_t(\phi,M) ) \to K_{s+t}(\phi, M)$$
defined by $a\otimes f \to a.f$. Restricting the domain of the map
to $K_s (\phi,Z_t(\phi, M))$ we get a map
$$ K_s (\phi, Z_t(\phi, M)) \to K_{s+t}(\phi, M)$$
which is indeed a map of complexes. So it induces a map
$$\alpha_t: Z_s (\phi, Z_t(\phi, M)) \to Z_{s+t}(\phi, M)$$
defined by
$$\sum a\otimes f \in Z_s (\phi, Z_t(\phi, M)) \to \sum a.f.$$

Now we define a map
$$\gamma_t: K_{s+t}(\phi, M) \to K_s (\phi, K_t(\phi, M))$$
by the formula
$$\gamma_t(f)= \sum_I e_I\otimes b_I$$
where the sum is over the $I\subset [n]$ with $\#I=s$ and $b_I$ is
determined by the decomposition (\ref{deco}). We claim:

\begin{lem}
\label{gamma} The map $\gamma_t: K(\phi,M)\to K(\phi,K_t(\phi,
M))[-t] $ is a map of complexes.
\end{lem}

\begin{proof}
Since $K(\phi,M)=K(\phi,R)\otimes M$ and we have $K(\phi,K_t(\phi,
M))=K(\phi,K_t(\phi, R))\otimes M$ it is enough to prove the
statement in the case $M=R$. Then it is enough to check
$$
\gamma_t\circ \phi(e_J)=\phi\circ \gamma_t(e_J)\quad \mbox{for every
} J\subset [n] \mbox{ with } \#J=s+t.
$$
Note that
$$\gamma_t(e_J)=\sum \sigma(A,B) e_A \otimes e_B$$
where the sum is over all the $B$ such that $\#B=t$ and
$A=J\setminus B$. Then
$$\phi\circ \gamma_t(e_J)=\sum \sigma(A\cup\{p\},B)\sigma(\{p\},A ) \phi(e_p) e_A\otimes e_B$$
and
$$ \gamma_t\circ \phi(e_J)=\sum \sigma(\{p\},A\cup B)\sigma(A,B ) \phi(e_p) e_A\otimes e_B$$
where in both cases the sum is over all the partitions of $J$ into
three parts $A,B,\{p\}$ with $\#B=t$. So we have to check that
$$\sigma(A\cup\{p\},B)\sigma(\{p\},A )= \sigma(\{p\},A\cup B)\sigma(A,B ).$$
This is a special case of \ref{checksign}.
\end{proof}

It follows that $\gamma_t$ gives, by restriction, a map
$$Z_{s+t}(\phi, M) \to Z_s (\phi, K_t(\phi, M)).$$
By virtue of \ref{le1}, its image is indeed contained in $Z_s (\phi,
Z_t(\phi, M))$. So we have a map
$$\beta_t: Z_{s+t}(\phi, M) \to Z_s (\phi, Z_t(\phi, M))$$
and, by virtue of Lemma \ref{le1}, we have
$$\alpha_t \circ \beta_t (f)=\binom{t+s}{s} f \ \ \ \
\mbox{ for all } f\in Z_{s+t}(\phi, M).$$

An immediate consequence:

\begin{lem}
\label{dirsum} Assume $\binom{t+s}{s}$ is invertible in $R$. Then
$Z_{s+t}(\phi, M)$ is a direct summand of $Z_s (\phi, Z_t(\phi,
M))$.
\end{lem}

One can easily check that, in the graded setting, the maps described in this section are graded and of degree $0$.

\section{Bounds for Koszul cycles}
\label{Bang}

In this section we consider a field $K$ and a standard graded
$K$-algebra $R$ with maximal homogeneous ideal $\mm$. In other
words, $R$ is of the form $S/J$ where $S$ is a polynomial ring over
$K$ with the standard grading and $J$ is a homogeneous of $S$. We
will consider a finitely generated graded $R$-module $M$.  Let
$\beta_{i,j}^R(M)=\dim_K \Tor_i^R(M,K)_j$ be the \emph{graded Betti
numbers} of $M$ over $R$. We define the number
$$
t_i^R(M)=\max\{j \in \Z : \beta^R_{i,j}(M)\neq 0\},
$$
whenever  $\Tor_i^R(M,K)\neq 0$ and $t_i^R(M)=-\infty$ otherwise.
The \emph{Castelnuovo-Mumford regularity}  of $M$ over $R$ is
$$
\reg_R(M)=\sup \{ t^R_i(M)-i : i\in \N \}.
$$
Recall that  $R$ is a \emph{Koszul  algebra} if $\reg_R(K)=0$. One
knows that $\reg_R(M)$ is finite for every finitely generated module
$M$ if $R$ is a Koszul algebra, see Avramov and Eisenbud  \cite{AE}.
One says that $R$ has the property $N_p$ if its defining ideal $J$
is generated by quadrics and the syzygies of the quadrics are linear
for $p-1$ steps, that is, if   $t_i^S(R)\leq i+1 \mbox{ for }
i=1,\dots,p$. The  \emph{Green-Lazarsfeld index}  of $R$ is the
largest number $p$ such that $R$ has the property $N_p$,  that is,
$$\ind(R)=\max\{ p : t_i^S(R)\leq i+1 \mbox{ for } i=1,\dots,p\}.$$
\medskip

\noindent{ \bf Conventions:} Just to avoid endless repetitions,
throughout this section ideals will be homogeneous, modules will be
finitely generated and graded, linear maps will be graded of degree
$0$. Furthermore $I$ will always denote an ideal and $M$ a module of
the current ring. The current ring will be denoted by $S$ if it is
the polynomial ring over a field $K$ or by $R$ if it is a standard
graded $K$-algebra and $\mm$ will denote its maximal homogeneous
ideal.
\medskip

We start with a well-known fact that is easy to prove:

\begin{lem}\label{annih}
One has  $$I+\Ann(M)\subseteq \Ann(M/IM)  \subseteq
\sqrt{I+\Ann(M)}.$$
\end{lem}

We have:

\begin{prop}\label{regb1}
Assume $R$ is   Koszul and  $\dim M/IM=0$. Let $c$ be the smallest
integer such that $\mm^c \subseteq  I+\Ann(M)$ and $I$ is generated
in degree $\leq c$ (such a number $c$ exists by \ref{annih}). Set
$Z_t=Z_t(I,M)$ and $H_t=H_t(I,M)$. Then, for every $t$,
\begin{align*}
& \reg_R (Z_t)\leq t(c+1)+\reg_R(M)\\ 
\intertext{and}\\
&\reg_R (H_t)\leq t(c+1)+\reg_R(M)+c-1.
\end{align*}
\end{prop}

\begin{proof}
The proof is a slight generalization of the arguments given in
\cite[Section 2]{BCR}.  Set $B_t=B_t(I,M)$  and  $K_t=K_t(I,R)$.
Note that $I+\Ann(M)$ annihilates $H_t$. Hence $\mm^cH_t=0$. It
follows that $H_t$ vanishes in degrees  $\geq t^R_0(Z_t)+c$ and
hence $\reg_R(H_t)\leq t_0^R(Z_t)+c-1\leq \reg_R(Z_t)+c-1$. So the
second formula follows from the first. The short exact sequence
$$0\to B_t\to Z_t\to H_t\to 0$$
gives $$\reg_R(B_t)\leq \max\{\reg(Z_t), \reg_R(H_t)+1\}\leq
\reg_R(Z_t)+c$$and $$0\to Z_{t+1}\to K_{t+1}\otimes M\to B_t\to 0$$
gives
\begin{align*}
 \reg_R(Z_{t+1})\leq &\max\{ \reg_R(K_{t+1}\otimes M), \reg_R(B_t)+1\} \\
                            \leq&\max\{ (t+1)c+\reg_R(M), \reg_R(Z_t)+c+1\}
\end{align*}
Now the statement can be proved by induction on $t$, the case $t=0$
being obvious since $Z_0=M$.
\end{proof}

We single out a special case of \ref{regb1}:

\begin{prop}\label{regb2}
Assume that $\dim S/I=0$.  Set $Z_t=Z_t(I,M)$ and $H_t=Z_t(I,M)$.
Then, for every $t$,
\begin{align*}
&\reg_S (Z_t)\leq t(\reg_S(I)+1)+\reg_S(M)\\
\intertext{and }\\
&\reg_S (H_t)\leq t(\reg_S(I)+1)+\reg_S(M)+\reg_S(I)-1.
\end{align*}

\end{prop}
\begin{proof} The number $c$ of $\ref{regb1}$ is $\leq \reg_S(I)$.\end{proof}

The following remark explains why the assumption on the dimension of
$S/I$ is necessary in \ref{regb2}.

\begin{rem}\label{rem1} The module $Z_1(I,M)$ sits in the
exact sequence:
$$0\to Z_1(I,M)\to F\otimes M\to IM\to 0.$$
Hence
$$\reg_S(IM)\leq \max\{ \reg_S(I)+\reg_S(M), \reg_S(Z_1(I,M))-1\}.$$
There are plenty of examples such that
$\reg_S(IM)>\reg_S(I)+\reg_S(M)$ already when $M=I$, see Conca
\cite{C} or Sturmfels \cite{S}. Therefore, in these examples, one
has $\reg_S(Z_1(I,M))>\reg_S(I)+1+\reg_S(M)$.
\end{rem}

But using a result of Caviglia \cite{CA}, see also Eisenbud, Huneke
and Ulrich \cite{EHU}, we are able to show:

\begin{thm}
\label{thm1}  Assume  that $\dim M/IM \leq 1$.
 Assume also that either $\chara K=0$ or $>t$.
Set $Z_t=Z_t(I,M)$. Then
$$\reg_S(Z_t)\leq t(\reg_S(I)+1)+\reg_S(M)$$
for every t.
\end{thm}

\begin{proof} By induction on $t$. For $t=1$ note that, by \cite{CA},
we have $\reg_S(IM)\leq \reg_S(I)+\reg_S(M)$ and the short exact
sequence of \ref{rem1} implies that $\reg_S(Z_1)\leq
\reg_S(I)+1+\reg_S(M)$. For $t>1$, by virtue of \ref{dirsum} we have
that $Z_t$ is a direct summand of $Z_{t-1}(I,Z_1)$. Hence
$\reg_S(Z_t)\leq \reg_S(Z_{t-1}(I,Z_1))$. Since $\Ann(Z_1)\supseteq
\Ann(M)$ we have $\Ann(Z_1)+I\supseteq \Ann(M)+I$ and, by \ref{annih},
$\dim Z_1/IZ_1\leq M/IM\leq 1$. Hence, by induction, we have
$$\reg_S(Z_{t-1}(I,Z_1))\leq (t-1)(\reg_S(I)+1)+\reg_S(Z_1).$$ Since $\reg_S(Z_1)\leq
\reg_S(I)+1+\reg_S(M)$ has been already established, the desired
inequality follows.
\end{proof}

\begin{quest}
\label{q1}
\begin{itemize}
\item[(1)] Does the inequality in \ref{thm1} hold over a Koszul algebra $R$?
And is the assumption on the characteristic needed?
\item[(2)] Is it true that $\reg_S(Z_t(I,S))\leq t(\reg_S(I)+1)$
holds for every homogeneous ideal $I\subset S$?
\end{itemize}
\end{quest}

Since $Z_1(I,S)$ is the first syzygy module of  $I$ the inequality
of   \ref{q1}  is actually an equality for $t=1$. An indication that
the  answer to \ref{q1}(2) might be ``yes'' for some classes of
ideals is given in \ref{piper} and \ref{sato}. Recall that a
monomial ideal $I\subset S=K[x_1,\dots,x_n]$ is \emph{strongly
stable}  if whenever a monomial $m\in I$ is divisible by a variable
$x_i$, then $mx_j/x_i\in I$ for every $j<i$.   In characteristic $0$
the strongly stable ideals are exactly the ideals of $S$ which are
fixed by the Borel group of the upper triangular matrices of
$\GL_n(K)$ acting on $S$.  The Eliahou-Kervaire complex  \cite{EK}
gives the graded  minimal free resolution of strongly stable ideals.
For us it is important to recall that if $I$ is strongly stable then
$\reg_S(I)$ is the largest degree of a minimal generator of $I$.

\begin{prop}
\label{piper} Let $I\subset S$ be a strongly stable ideal. Set
$Z_t=Z_t(I,S)$. Then $Z_t$ is generated by elements of degree $\leq
t(\reg_S(I)+1)$.
\end{prop}

\begin{proof}
Set $c=\reg_S(I)$. The idea of the proof follows essentially the
argument given in \cite[Theorem 3.3]{BCR}. We note first that, as we
are dealing with a monomial ideal $I$, the modules $Z_t$ have a
natural $\Z^n$-graded structure as long as we consider the free
presentation $F\to I$ associated with the monomial generators of
$I$. We do a double induction on $n$ and on $t$. The case $n=1$ is
obvious. The case $t=1$ is easy and follows from the description of
the (first) syzygies of $I$ given in \cite{EK}. By induction on $t$
it is enough to verify that $Z_t /Z_1 Z_{t-1}$ is generated in
degree $<t(c+1)$. Hence it suffices to show that every $\Z^n$-graded
element $f\in Z_t$ of total degree $q\geq t(c+1)$ can be written
modulo $Z_1Z_{t-1}$ as a multiple of an element in $Z_t$ of total
degree $<q$. Let $\alpha\in \Z^n$ be the $\Z^n$-degree of $f$. If
$\alpha_n=0$ then we can conclude by induction on $n$. Therefore we
may assume that $\alpha_n>0$. Let $u\in I$ be a monomial generator
of $I$ with $x_n\mid u$ and $[u]$ the corresponding free generator
of $F$. We have the decomposition
$$f=a+[u] b$$
with $b\in Z_{t-1}$ and $[u]$ does not appear in $a$. Note that $b$
has degree $q-\deg(u)\geq q-c$. Since $Z_{t-1}$ is generated by
elements of degree $\leq (t-1)(c+1)$ we may write
\begin{equation}
\label{eq:1} b=\sum_{j=1}^s \lambda_j v_j z_j
\end{equation}
where $\lambda_j\in K$, $z_j\in Z_{t-1}$ are $\Z^n$ graded and the
$v_j$ are monomials of positive degree.

Let $\lambda_j v_j z_j$ be a summand in (\ref{eq:1}). If $x_n$ does
not divide $v_j$, then choose $i<n$ such that $x_i\mid v_j$. Since
$x_iu/x_n\in I$, there exists a monomial generator of $I$, say
$u_1$, such that $u_1\mid x_iu/x_n$, say $u_1w=x_iu/x_n$. Set
$z'=x_i[u]-x_nw[u_1]\in Z_1$ and subtract the element
$$
\lambda_j \frac{v_j}{x_i} z'z_j \in Z_{t-1}Z_1
$$
from $f$. Repeating this procedure for each $\lambda_j v_j z_j$ in
(\ref{eq:1}) such that $x_n$ does not divide $v_j$ we obtain a cycle
$f_1\in Z_t$ of degree $\alpha$ such that
\begin{itemize}
\item[(i)] $f\equiv f_1 \mod Z_1Z_{t-1}$;
\item[(ii)] if $v [u_1,\dots,u_t]$ appears in $f_1$
and $u\in \{ u_1,\dots,u_t\}$, then $x_n\mid v$.
\end{itemize}
We repeat the described procedure for each monomial generator $u\in
I$ with $x_n\mid u$. We end up with an element $f_2\in Z_t$ of
degree $\alpha$ such that
\begin{itemize}
\item[(iii)] $f\equiv f_2 \mod Z_1Z_{t-1}$;
\item[(iv)] if  $v [u_1,\dots,u_t]$ appears in $f_2$
and $x_n\mid u_1\cdots u_t$, then $x_n\mid v$.
\end{itemize}
Note that if $v [u_1, \dots, u_t]$ appears in $f_2$ and $x_n\nmid
u_1\cdots u_t$, then $x_n\mid v$ by degree reasons. Hence for every
$v [u_1, \dots, u_t]$ appearing in $f_2$ we have $x_n\mid v$.
Therefore $f_2=x_ng$, and $g \in Z_t$ has degree $<q$. This
completes the proof.
\end{proof}

 Indeed a much stronger statement holds:

 \begin{thm}
 \label{sato}
 Let $I,J$ be strongly stable ideals of $S$. Then
 $$\reg_S(Z_t(I,S/J))\leq t(\reg_S(I)+1)+\reg_S(S/J)$$
 for every $t$.
\end{thm}

Theorem \ref{sato} has been proved by Satoshi Murai in collaboration
with the second author and is part of an ongoing project.

The following result, whose proof is surprisingly simple,
generalizes Green's theorem \cite[Theorem 2.16]{GR2}:

\begin{thm}\label{greeny}
Let $I\subset S$  such that  $\dim M/IM=0$. Let $c\in \N$ be such
that $I$  is generated in degrees $\leq c$ and set $v=\dim [S/I]_c$.
Set $Z_t=Z_t(I,M)$ and $H_t=H_t(I,M)$. One has
\begin{align*}
&\reg_S(Z_t)\leq t(c+1)+\reg_S(M)+v \\
\intertext{and}  \\
&\reg_S(H_t)\leq t(c+1)+\reg_S(M)+v+c-1
\end{align*}
for every $t$.
\end{thm}
\begin{proof} The first inequality can be deduced from the second using
the standard short exact sequences relating $B_t, Z_t$ and $H_t$.
We prove the second inequality by induction on $v$. If $v=0$ then
$\mm^c\subset I$ and the assertion has been proved in \ref{regb1}.
Now let $v>0$. Observe that $H_t$ is annihilated by $I+\Ann(M)$.
Hence (by \ref{annih})  $\dim H_t=0$ and its regularity is the
largest degree  in which $H_t$ does not vanish. Take $f\in
S_c\setminus I$ and set $J=I+(f)$. Note that the minimal generators
of $I$ are minimal generators of $J$ and that $\dim [S/J]_c=v-1$. We
have a short exact sequence of Koszul homology \cite[1.6.13]{BH}:
$$H_{t+1}(J,M)\to H_{t}(-c)\to  H_{t}.$$
By construction $H_{t}(-c)$ does not vanish in degree
$\reg_S(H_t)+c$ while  $H_t$  vanishes in that degree. It follows
that $H_{t+1}(J,M)$ does not vanish in degree $\reg_S(H_t)+c$ and
hence $\reg_S(H_{t+1}(J,M))\geq \reg_S(H_t)+c$.  By induction we
know that $\reg_S(H_{t+1}(J,M))\leq
(t+1)(c+1)+\reg_S(M)+(v-1)+c-1$. It follows that
$$\reg_S(H_t)+c\leq  (t+1)(c+1)+\reg_S(M)+(v-1)+c-1, $$
that is, \begin{equation*} \reg_S(H_t)\leq
t(c+1)+\reg_S(M)+v+c-1.\qedhere\end{equation*}
\end{proof}

\begin{rem} (a) Let $I\subset S$ be the ideal generated by a proper
subspace $V$ of forms of degree $c$  such that $I\mm=\mm^{c+1}$. Then
$\reg_S(I)=c+1$. Set $Z_t=Z_t(I,S)$ and $v=\dim S_c/V$. By virtue of
\ref{regb2} we have $\reg_S(Z_t)\leq t(c+2)$ while \ref{greeny}
gives $\reg_S(Z_t)\leq t(c+1)+v$.  So for small $t$ the first bound
is better than the second and the other way round for large $t$.

(b)   Since $H_0=M/IM$, for $t=0$ the bound of \ref{greeny} takes
the form $\reg_S(M/IM)\leq \reg_S(M)+v+c-1$.  Even the case $M=S$ is
interesting: it says that if $\sqrt{I}=\mm$, $I$ is generated in
degree $\leq c$ and $v=\dim [S/I]_c$ then $\mm^{c+v}\subset I$.
\end{rem}

\section{Green-Lazarsfeld index for Segre-Veronese rings}
\label{SeVe}

The goal of this section is to prove a  result  \ref{check}    about
the Green-Lazarsfeld index of Segre-Veronese rings which was
announced in \cite{BCR}.  We first need to generalize some results
of \cite{BCR} to the multihomogeneous setting.

Let $d\in \N$ and $m=(m_1,\dots,m_d)\in \N^d$ and
$\fc=(c_1,\dots,c_d)\in \N^d$. We consider the polynomial ring
$S=K[x_{ij} : 1\leq i\leq d, \ \ 1\leq j\leq m_i]$ with the $\Z^d$
graded structure induced by assigning $\deg x_{ij}=e_i\in \Z^d$.
Consider the ideals $\mm_i=(x_{ij} | j=1,\dots,m_i)$ and
$$\mm^{\fc}=\prod_{i=1}^d \mm_i^{c_i}.$$
Then the module of Koszul cycles $Z_t(\mm^{\fc},S)$ has a
$\Z^d$-graded structure and also a finer $\Z^m=\Z^{m_1}\times \cdots
\Z^{m_d}$-graded structure. We have:

\begin{lem}
\label{multi} The module $Z_t(\mm^{\fc},S)$ is generated by elements
that either have $\Z^d$-degree bounded above by the vector
$t\fc+(t-1)\sum e_i$ or belong to $U_i^t$ for some $i$ in
$\{1,\dots,d\}$ where $U_i=Z_1(\mm^{\fc},S)_{\fc+e_i}$.
\end{lem}

\begin{proof} Set $Z_t=Z_t(\mm^{\fc},S)$ and give it the natural
$\Z^m$ graded structure. The proof is a multigraded variant of the
argument used above in \ref{piper}. First note that given a monomial
generator $u$ of $\mm^{\fc}$, a variable $x_{ij}|u$  and $k$ such
that $1\leq k\leq m_i$ and $k\neq j$, the monomial
$v=ux_{ik}/x_{ij}$ belongs to $\mm^{\fc}$ and the element
$x_{ik}[u]-x_{ij}[v]$ belongs to $U_i$. It is well known that these
syzygies generate $Z_1$ and that $\mm^{\fc}$ has a linear
resolution. Now assume that $f\in Z_t$ is a $\Z^m$-homogeneous
element of degree $(\alpha_1,\dots,\alpha_d)$,
$\alpha_i=(\alpha_{i1},\dots,\alpha_{im_i})\in \Z^{m_i}$. Assume
that $|\alpha_i|\geq t(c_i+1)$ for some $i$, say for $i=1$. We may
also assume that $\alpha_{11}\neq 0$. Using induction on $t$, the
rewriting procedure described in the proof of \ref{piper} and the
linear syzygies described above, we may write $f=x_{11}g
\mod{U_1Z_{t-1}}$. Since $g\in Z_t$, the conclusion follows by
induction on $t$.
\end{proof}

Next we note that \cite[Lemma 3.4]{BCR} can be extended to the
present setting:

\begin{lem}
\label{newcycles} Let $\alpha\in \N^d$ be a vector such that
$\alpha\leq {\fc}$ componentwise. Let $a_1,a_2\dots,a_{t+1}$ be
monomials of $\Z^d$-degree equal to $\alpha$ and
$b_1,b_2\dots,b_{t}\in S$ monomials of degree ${\fc}-\alpha$. Then
\begin{equation}
\label{eq:cycle} \sum_{\sigma\in \mathcal{S}_{t+1}} (-1)^{\signum
\sigma} a_{\sigma(t+1)} [b_1a_{\sigma(1)}, b_2a_{\sigma(2)}, \dots,
b_ta_{\sigma(t)}]
\end{equation}
belongs to $Z_t(\mm^{\fc},S)$.
\end{lem}

Now we prove a multigraded version of \cite[Thm.3.6]{BCR}:

\begin{lem}
\label{multi2} For every $i=1,\dots,d$ let ${\ffb}={\fc}-e_i$ and
set $U_i=Z_1(\mm^{\fc},S)_{\fc+e_i}$. Then
$$(c_i+1)!\, \mm^{\ffb} U_i^{c_i}\subset \mm_i^{c_i} Z_{c_i}(\mm^{\fc},S)+B_{c_i}(\mm^{\fc},S).$$
\end{lem}

\begin{proof} Set $u=c_i$ and $Z_u=Z_{c_i}(\mm^{\fc},S)$ and $B_u=B_{c_i}(\mm^{\fc},S)$.
The generators of $U_i$ are of the form
$$z_a(y_0,y_1)=y_0[ay_1]-y_1[ay_0]$$ where
$a$ is a monomial of $\Z^d$-degree equal to ${\ffb}$ and $y_0,y_1\in
\{x_{i1}, \dots, x_{im_i}\}$. So we have to take $u$ such elements,
say $z_{a_j}(y_{0j},y_{1j})$ with $j=1,\dots,u$, another monomial of
degree ${\ffb}$, say $a_{u+1}$, and we have to prove that
\begin{equation}
(u+1)! a_{u+1}\prod_{j=1}^{u} z_{a_j}(y_{0j},y_{1j}) \in \mm_i^{u}
Z_{u}+B_{u}.
\end{equation}
\label{superg} The symmetrization argument given in the proof of
\cite[Theorem 3.6]{BCR} works in this case as well to prove that the
left hand side of (\ref{superg}) can be rewritten, modulo
boundaries, as
$$\sum y_{i_11}\cdots y_{i_{u},u} W_{i}$$
where $i=(i_1,\dots,i_u)\in \{0,1\}^u$ and $W_{i}$ are cycles of the
type described in \ref{newcycles}.
\end{proof}

An $\N^d$-graded $K$-algebra $R=\bigoplus_{\alpha \in \N^d}
R_\alpha$ is called standard if $R_0=K$ and $R$ is generated by
$R_{e_i}$ with $i=1,\dots,d$. Clearly $R$ can be presented as a
quotient of an $\N^d$-graded polynomial ring $S=K[x_{ij} : 1\leq
i\leq d, \ \ 1\leq j\leq m_i]$ with the $\N^d$-graded structure
induced by assigning $\deg x_{ij}=e_i\in \Z^d$. Given a vector
$${\fc}=(c_1,\dots,c_d)\in \N^d,$$
we can consider the \emph{Segre-Veronese subring} of $R$ associated
to it, namely
$$R^{(\fc)}=\bigoplus_{j\in \N} R_{j{\fc}}.$$

Our goal is to study the Green-Lazarsfeld index  of $R^{(\fc)}$. We
note that $R^{(\fc)}$ is a quotient ring of $S^{(\fc)}$. Furthermore
one has:
\begin{lem} \begin{itemize}
\item[(a)] $\reg_{S^{(\fc)}}(R^{(\fc)})=0$ for ${\fc}\gg 0$.
\item[(b)] $\ind(R^{(\fc)})\geq \ind(S^{(\fc)})$ for ${\fc}\gg 0$.
\end{itemize}
More precisely, both statements hold provided one has $i{\fc}\geq
\alpha$ componentwise for every $i\in \Z$ and $\alpha\in \Z^d$ such
that $\beta_{i,\alpha}^S(R)\neq 0$.
\end{lem}

\begin{proof}
A detailed proof of (a) is given in the bigraded setting by Conca,
Herzog, Trung and Valla \cite{CHTV}. The same argument works as well
for multigradings. Then (b) follows from (a) and \cite[Lemma
2.2]{BCR}.
\end{proof}

Consider the symmetric algebra $T$ of the $K$-vector space $S_{\fc}$
(i.e. a polynomial ring of Krull dimension $\dim_K S_{\fc}$), and
the natural surjection $T\to S^{(\fc)}$.  The Betti numbers of
$S^{({\fc})}$ as a $T$-module can be computed via Koszul homology.

\begin{lem} We have $$\beta^T_{i,j}(S^{(\fc)})=H_i(\mm^{\fc},S)_{j{\fc}}.$$
\end{lem}
\begin{proof}
One notes that $S^{(\fc)}$ is a direct summand of $S$  and then
proceeds as in \cite[Lemma 4.1]{BCR}
 \end{proof}

So we may reinterpret  \ref{multi}   in terms of syzygies of
$S^{(\fc)}$, obtaining:
\begin{cor}
\label{multihomo} One has $\beta^T_{i,j}(S^{(\fc)})=0$ provided
$(j-i-1)\min(\fc)\geq i$.
\end{cor}
\begin{proof}
Since $\mm^{\fc}$ annihilates $H_i(\mm^{\fc},S)$, it follows from
\ref{multi} that $H_i(\mm^{\fc},S)_{\alpha}=0$ for every $\alpha\geq
i(\fc+\sum e_i)+\fc$ componentwise. Replacing $\alpha$ with $j\fc$
we have that $\beta^T_{i,j}(S^{(\fc)})=0$ if $j\fc\geq i(\fc+\sum
e_i)+\fc$ which is equivalent to $(j-i-1)\min(\fc)\geq i$.
\end{proof}

In \cite{HSS} Hering, Schenck and Smith proved that
$\ind(S^{(c)})\geq \min(c)$. We improve the bound by one:
\begin{thm}
\label{check} One has $\min(\fc)\leq \ind(S^{(\fc)})$. Moreover,
$\min(\fc)+1\leq \ind(S^{(\fc)})$ if $\chara K=0$ or $\chara
K>1+\min(c)$.
\end{thm}
\begin{proof}
The first statement is an immediate consequence of \ref{multihomo}.
In fact, if $i\leq \min(\fc)$ then $(j-i-1)\min(\fc)\geq i$ for
every $j>i+1$ and hence, by \ref{multihomo}, $t_i^T(S^{(\fc)})=i+1$.
Set $u=\min(\fc)$. For the second statement, we have to show that
$H_{u+1}(\mm^{\fc},S)_{j\fc}=0$ for every $j>u+2$. By virtue of
\ref{multi} we know that $Z_{u+1}(\mm^{\fc},S)$ is generated by :

\begin{itemize}
\item[(1)] elements of degree $\leq (u+1)\fc+u\sum_s e_s$ and
\item[(2)]  elements of $U_i^{u+1}$ where $U_i=Z_1(\mm^{\fc},S)_{\fc+e_i}$
and $i=1,\dots,d$; they have degree $(u+1)\fc+(u+1)e_i$.
\end{itemize}
So an element $f\in Z_{u+1}(\mm^{\fc},S)_{j\fc}$ can come from a
generator of type (1) by multiplication of elements of degree
$\alpha\in \Z^d$ such that
$$\alpha\geq j\fc-(u+1)\fc+u\sum_s e_s.$$
Since $j>u+2$, we have
$$\alpha \geq 2\fc+u\sum_s e_s\geq \fc. $$
So $f\in \mm^{\fc} Z_{u+1}(\mm^{\fc},S)$ and $f=0$ in homology.
Alternatively, $f\in Z_{u+1}(\mm^{\fc},S)_{j\fc}$ can come from a
generator of type (2) by multiplication of elements of degree
$\alpha\in \Z^d$ such that
$$\alpha= j\fc-(u+1)\fc-(u+1)e_i\geq 2\fc-(u+1)e_i.$$
If $c_i>u$ then $\alpha\geq \fc$, and we conclude as above that
$f=0$ in homology. If, instead, $c_i=u$, then $\alpha\geq
2\fc-(c_i+1)e_i$. We have that
$$f\in \mm^{2\fc-(c_i+1)e_i}U_i^{c_i+1}= \mm^{\fc-c_ie_i}\mm^{\fc-e_i}U_i^{c_i}U_i.$$
But, assuming $K$ has either characteristic $0$ or $>u+1$,
\ref{multi2} implies:
$$\mm^{\fc-e_i}U_i^{c_i}\subset \mm_i^{c_i}Z_{c_i}(\mm^{\fc},S)+B_{c_i}(\mm^{\fc},S).$$
Hence
$$f\in \mm^{\fc}Z_{u}(\mm^{\fc},S)U_i+B_{c_i}(\mm^{\fc},S)U_i\subset B_{c_i+1}(\mm^{\fc},S)$$
and we conclude that $f=0$ in homology.
\end{proof}

\section{Generating Koszul cycles}
\label{Gency}

In this section we consider the Koszul cycles $Z_t(I,R)$ where $R$
is standard graded and $I$ is a homogeneous ideal. For simplicity,
in this section we let $Z_t$ denote the cycles $Z_t(I,R)$, and
similarly write $B_t$, $H_t$ and $K_t$ for boundary, homology and
components of the Koszul complex $K(I,R)$. We consider the
multiplication map
\begin{equation}
\label{mianpo} Z_{s} \otimes Z_{t} \to Z_{s+t}
\end{equation}
and we want to understand in which degrees it is surjective. Note
that the map (\ref{mianpo}) has a factorization
$$Z_{s} \otimes Z_{t} \overset{u_{s,t}} {\longrightarrow} Z_s (I, Z_t )
\overset{\alpha_t}{\longrightarrow} Z_{s+t}$$ where the first map
$u_{s,t}$ is the canonical one and the second is the map $\alpha_t$
described in Section \ref{Canonical maps}.

\begin{prop}
\label{surge} Suppose $R$ has characteristic $0$ or larger than
$s+t$. Then:
\begin{itemize}
\item[(1)] The multiplication map $Z_{s}\otimes Z_{t} \to Z_{s+t} $ is
surjective in degree $j$ if the module $\Tor_1^R(K_{s-1}/B_{s-1},
Z_t)$ vanishes in degree $j$.
\item[(2)] If $R=K[x_1,\dots,x_n]$ and $\dim R/I=0$ then the
multiplication map $Z_{s}\otimes Z_{t} \to Z_{s+t}$ is surjective in
degree $j$ for every $j\geq \reg_R Z_s+\reg_R Z_t$. In particular,
the map $Z_{s}\otimes Z_{t} \to Z_{s+t}$ is surjective in degree $j$
for every $j\geq (s+t)(\reg_R(I)+1)$.
\end{itemize}
\end{prop}

\begin{proof} To prove (1) we note that, as $\alpha_t$ is surjective,
we may as well consider the map $u_{s,t}:Z_{s}\otimes Z_{t} \to Z_s
(\phi, Z_t)$. Tensoring
$$0\to Z_s\to K_s\to B_{s-1}\to 0$$
and
$$0\to B_{s-1}\to K_{s-1}\to K_{s-1}/B_{s-1}\to 0$$
 with $Z_t$, we have exact sequences
$$Z_s\otimes Z_t\to K_s\otimes Z_t \overset{f}{\to} B_{s-1}\otimes Z_t\to 0$$
and
$$\Tor_1^R(K_{s-1}/B_{s-1}, Z_t) \to B_{s-1}\otimes Z_t\overset{g}{\to} K_{s-1}\otimes Z_t.$$
The composition $g\circ f$ is the map of the Koszul complex
$K_s\otimes Z_t \to K_{s-1}\otimes Z_t$. So $Z_s(\phi,Z_t)=\ker
(g\circ f)$ and the image of $u_{s,t}$ is $\ker f$. It follows that
$u_{s,t}$ is surjective in degree $j$ iff $g$ is injective in degree
$j$, that is $\Tor_1^R(K_{s-1}/B_{s-1}, Z_t)$ vanishes in degree
$j$.

To prove (2) we first  observe, since $\sqrt{I}=\mm$,  one has that
$(Z_t)_P$  is free for every prime ideal $P\neq \mm$. Hence
$\Tor_i^R(M,Z_t)$ has Krull dimension $0$ for every finitely
generated $R$-module $M$ and every $t\geq 0$ and $i>0$.

Then  we may apply \cite[Corollary 3.1]{EHU} and have that
$$\reg_R \Tor^R_i(M,Z_t)\leq \reg_R M+\reg_R Z_t+i$$ and in particular
$$\reg_R \Tor^R_1(K_{s-1}/B_{s-1},Z_t)\leq \reg_R K_{s-1}/B_{s-1} +\reg_R Z_t+1.$$
But $ \reg_R K_{s-1}/B_{s-1}=\reg_R Z_s-2$ and hence
$$\reg_R \Tor^R_1(K_{s-1}/B_{s-1},Z_t)\leq \reg_R Z_s +\reg_R Z_t-1$$
In other words,  $\Tor^R_1(K_{s-1}/B_{s-1},Z_t)$ vanishes in degrees
$\geq \reg_R Z_s +\reg_R Z_t$. Together with (1) this concludes the
proof of (2).
\end{proof}

\begin{thm}
\label{maincyc} Assume that $R$ is Koszul and $K$ has characteristic
$0$ or $>t$ and take $I=\mm^c$. Then for every $t$ the module
$Z_t/Z_1^t$ vanishes in degree $\geq t(c+1)$ and $Z_1^t$ has an
$R$-linear resolution.
\end{thm}

\begin{proof} We prove the first assertion by induction on $t$.
It is enough to prove that the multiplication map $Z_1\otimes
Z_{t-1} \to Z_t$ is surjective in degrees $j\geq tc+t$. By virtue of
\ref{surge}(1), it is enough to prove that $\Tor_1^R(R/\mm^c,
Z_{t-1})$ vanishes for $j\geq tc+t$. But since $R/\mm^c$ vanishes in
degree $\geq c$, it is easy to see that $\Tor_1^R(R/\mm^c, Z_{t-1})$
vanishes in degrees $\geq t_1^R(Z_{t-1})+c$. We know
\cite[Proposition 2.4]{BCR} that $\reg_R(Z_i)\leq ic+i$ for every
$i$ (here we use the fact that $R$ is Koszul). So we have
$t_1^R(Z_{t-1})-1\leq (t-1)c+(t-1)$, i.e. $t_1^R(Z_{t-1})\leq
(t-1)c+t$. Then we have $t_1(Z_{t-1})+c\leq (t-1)c+t+c=tc+t$. This
proves the first assertion. For the second, one just notes that
$\reg_R(Z_t)\leq tc+t$ and that $Z_1^t$ coincides with $Z_t$
truncated in degree $t(c+1)$. Therefore $Z_1^t$ must have an
$R$-linear resolution.
\end{proof}

We have the following consequence:

\begin{cor}
\label{prcpt} Let $S=K[x_1,\dots,x_n]$ with $\chara K=0$ or $>2c$.
One has  $H_1(\mm^c,S)^{2c}=0$  iff  $\ind(S^{(c)})\geq 2c$, i.e.
$S^{(c)}$ has the $N_{2c}$-property.
\end{cor}

\begin{proof} By virtue of \ref{maincyc} $Z_{2c}(\mm^c,S)$ coincides
with $Z_1(\mm^c,S)^{2c}$ in degrees $\geq 2c(c+1)$.
Hence,  by assumption,  $H_{2c}(\mm^c,S)$ vanishes in degrees $\geq
2c(c+1)$. This implies that $\beta^{T}_{2c,j}(S^{(c)})=0$ if $jc\geq
2c(c+1)$, that is, $j\geq  2c+2$. In other words,
$t_{2c}^T(S^{(c)})\leq 2c+1$. Since $S^{(c)}$ is Cohen-Macaulay, one
can conclude that $t_{i}^T(S^{(c)})\leq i+1$ for $i=1,\dots,2c$,
that is,  $\ind(S^{(c)})\geq 2c$.
\end{proof}

\begin{rem}The interesting aspect of Corollary \ref{prcpt} is that we know
explicitly the generators of $Z_1(\mm^c,S)$ and hence the inclusion
$Z_1(\mm^c,S)^{2c} \subset B_{2c}$ boils down to a quite concrete
statement. Unfortunately we have not been able to settle it. Note
also that Ottaviani and Paoletti conjectured that
$\ind(S^{(c)})=3c-3$ apart from few known exceptions and at least in
characteristic $0$, see \cite{OP} or \cite{BCR} for the precise
statements.  In \cite{BCR} we have proved that $\ind(S^{(c)})\geq
c+1$.
\end{rem}

As we mentioned in \cite{BCR} there are computational evidences that
the cycles of \cite[Lemma 3.4]{BCR} generate $Z_t(\mm^c,S)$. We show
below that this is the case for $t=2$ and any $c$. To this end we
recall that for every monomial $b$ of degree $c-1$ and for variables
$x_j,x_k$ we have an element $z_b(x_j,x_k)=x_j[bx_k]-x_k[bx_j]\in
Z_1(\mm^c,S)$. It is well-know and easy to see that the elements
$z_b(x_j,x_k)$ generate $Z_1(\mm^c,S)$. For a monomial $a$ we set
$\max(a)=\max\{ i : x_i|a\}$ and $\min(a)=\min\{ i : x_i|a\}$ . More
precisely,  the elements $z_b(x_j,x_k)$ with $j<k$ and $\max(b)\leq
k$ form a Gr\"obner basis of $Z_1(\mm^c,S)$ with respect to any term
order selecting $x_j[bx_k]$ as leading term of $z_b(x_j,x_k)$. We
have:
\begin{prop}
\label{gen2} If $K$ has characteristic $\neq 2$ then the module
$Z_2(\mm^c,S)$ is generated by two types of elements:
\begin{itemize}
\item[(1)] The elements of \cite[Lemma 3.4]{BCR} of degree $2c+1$,
\item[(2)] and by the elements of $Z_1(\mm^c,S)^2$ of degree $2c+2$, that is,
the elements of the form $z_a(x_i,x_j)z_b(x_h,x_k)$.
\end{itemize}
\end{prop}

\begin{proof}
Consider the map
$$\alpha_1: Z_1(\mm^c, Z_1(\mm^c,S))\to Z_2(\mm^c,S)$$
of Section \ref{Canonical maps}. We know that $Z_2(\mm^c,S)$ has
regularity $\leq 2c+2$ and the only generators of degree $2c+2$ are
the elements of (2). So we only need to deal with the elements of
degree $2c+1$.  To this end we look at the component of degree
$2c+1$ of $\alpha_1$. Let $a,b$ be monomials of degree $c-1$. The
element
\begin{equation}
\label{cycy} [ax_i]\otimes z_b(x_j,x_k)+ [ax_k]\otimes
z_b(x_i,x_j)+[ax_j]\otimes z_b(x_k,x_i)
\end{equation}
belong to $Z_1(\mm^c, Z_1(\mm^c,S))$ and has degree $2c+1$. The
image under $\alpha_1$ of the elements in (\ref{cycy}) are exactly
the cycles of \cite[Lemma 3.4]{BCR} in $Z_2(\mm^c,S)$. Since
$\alpha_1$ is surjective, to complete the proof it is enough to
prove the following statement: \medskip

\noindent {\bf Claim}: The cycles described in (\ref{cycy}) generate
$Z_1(\mm^c, Z_1(\mm^c,S))$ in degree $2c+1$.
\medskip

Let $F\in Z_1(\mm^c, Z_1(\mm^c,S))$ be an element of degree $2c+1$.
So $F$ is a sum of elements of the form $[u]\otimes f$ with $u$ a
monomial of degree $c$ and $f\in Z_1(\mm^c,S)$ with $\deg(f)=c+1$.
Choose $[u]$ to be the largest in the lexicographic  order induced
by $x_1>\dots>x_n$ and look at the coefficient $f$ of $[u]$ in $F$,
i.e.
$$F=[u]\otimes f + \mbox{ sum of terms } [v]\otimes g \mbox{ with } v<u.$$
Let $x_j[bx_k]$ be the leading term of $f$ with $j<k$ and
$\max(b)\leq k$. If $\min(u)<j$ then we may add a suitable scalar
multiple of (\ref{cycy}) to ``kill" the leading term of $F$ and we
are done. If instead $\min(u)\geq j$, then, since
$$0=\phi(F)=u f+ \mbox{ sum of terms } v g \mbox{ with } v<u$$
we have that $x_ju[bx_k]$ must cancel, and so $x_ju=x_sv$ for some
$v<u$ in the lex-order. But this is impossible.
\end{proof}

\end{document}